\documentclass[a4paper,11pt]{amsart}
 \usepackage{amssymb,amsfonts,amsmath}

\def\C{\mathbb{C}}
\def\c2{\mathbb{C}^2}
\def\R{\mathbb{R}}

\def\Z{\mathbb{Z}}

\def\P{\mathbb{P}}

\def\E{\mathcal{E}}
\def\W{\mathcal{W}}

\def\F{\mathbb{F}}
\def\1{\mathbf{1•}}

\def\a{\alpha}

\def\e{\varepsilon}
\def\l{\lambda}
\def\f{\varphi}

\def\om{\omega}

\newtheorem{Theorem}{Theorem}[section]
\newtheorem{Proposition}[Theorem]{Proposition}
\newtheorem{Lemma}[Theorem]{Lemma}
\newtheorem{Corollary}[Theorem]{Corollary}
\theoremstyle{definition}
\newtheorem{Definition}[Theorem]{Definition}
\newtheorem{Example}[Theorem]{Example}
\newtheorem{Examples}[Theorem]{Examples}
\newtheorem{Remark}[Theorem]{Remark}

\newtheorem*{ackn}{Acknowledgement}

\newcommand{\db}{\overline\partial}
\newcommand{\ov}{\overline}
\newcommand{\wi}{\widetilde}
\DeclareMathOperator{\vol}{vol}

\DeclareMathOperator{\inte}{int}
\DeclareMathOperator{\dist}{dist}

\newcommand{\comment}[1]{}

\begin{document}

\title[Toric pluripotential theory]{Toric pluripotential theory}

\author{Dan Coman, Vincent Guedj, Sibel Sahin, Ahmed Zeriahi}

 \date{April 9, 2018}
 \thanks{Dan Coman  is partially supported by the NSF Grant DMS-17000111}
 \thanks{Vincent Guedj and Ahmed Zeriahi are partially supported by the ANR project GRACK}
 \thanks{ Sibel Sahin is supported by the TUBITAK 2219 postdoctoral grant}

\address{Department of Mathematics, Syracuse University, Syracuse, NY 13244-1150, USA}
\email{dcoman@syr.edu}

\address{Institut de Math\'ematiques de Toulouse  \\
Universit\'e de Toulouse, CNRS \\
UPS, 118 route de Narbonne \\
31062 Toulouse cedex 09, France}
\email{vincent.guedj@math.univ-toulouse.fr}

\address{Department of Mathematics, Mimar Sinan Fine Arts University, Istanbul, Turkey}

\email{sibel.sahin@msgsu.edu.tr}

\address{ Institut de Math\'ematiques de Toulouse  \\
Universit\'e de Toulouse, CNRS \\
UPS, 118 route de Narbonne \\
31062 Toulouse cedex 09, France}
\email{ahmed.zeriahi@math.univ-toulouse.fr}

\maketitle

\begin{center}
{\it A tribute to Professor J\'ozef SICIAK}
\end{center}

\begin{abstract}
We study finite energy classes of quasiplurisubharmonic (qpsh) functions in the setting of toric compact K\"{a}hler manifolds. We characterize toric qpsh functions and give necessary and sufficient conditions for them
to have finite (weighted) energy, both in terms of the associated convex function in $\R^n$, and through
the integrability properties of its Legendre transform. We characterize Log-Lipschitz convex functions on the Delzant polytope, showing that they correspond to toric qpsh functions which satisfy a certain exponential integrability condition. In the particular case of dimension one, those Log-Lipschitz convex functions of the polytope correspond to H\"{o}lder continuous toric quasisubharmonic functions.
\end{abstract}

\tableofcontents


\section*{Introduction}
\label{sec:intro}

A {\it toric} compact K\"ahler manifold $(X,\omega,T)$ is an equivariant compactification of the torus $T=(\C^*)^n$ equipped with a $(S^1)^n$-invariant K\"ahler metric $\omega$. Then $\omega$ can be written as
$$\omega=dd^cF_0\circ L \text{ in } (\C^*)^n,$$
where $F_0:\R^n \rightarrow \R$ is a smooth strictly convex function and
\begin{equation}\label{e:L}
L:(\C^\star)^n\to\R^n\,,\,\;L(z_1,\ldots,z_n)=(\log|z_1|,\ldots,\log|z_n|)\,.
\end{equation}

The celebrated Atiyah-Guillemin-Sternberg theorem asserts that the moment map
$\nabla F_0:\R^n \rightarrow \R^n$ sends $\R^n$ onto the interior of a compact convex polytope
$$
P=\{ \ell_i(s) \geq 0, \; 1 \leq i \leq d \} \subset \R^n,
$$
where $d \geq n+1$ is the number of $(n-1)$-dimensional faces of $P$ and
$$
\ell_i(s)=\langle s,u_i \rangle -\lambda_i,
$$
with $\l_i \in \R$ and $u_i$ a primitive element of $\Z^n$.

Delzant observed in \cite{Del88} that in this case $P$ is ``Delzant", i.e. there are exactly
$n$ faces of dimension $(n-1)$ meeting at each vertex, and the corresponding $u_i$'s form a
$\Z$-basis of $\Z^n$. He conversely showed that there is exactly one (up to symplectomorphism)
toric compact K\"ahler manifold $(X_P,\{\omega_P\},T)$ associated to a Delzant polytope $P \subset \R^n$.
Here $\{ \omega_P\}$ denotes the cohomology class of the $T$-invariant K\"ahler form  $\omega_P$.

Let
$$
G_0(s):=\sup_{x \in \R^n} (\langle x,s \rangle-F_0(x))
$$
denote the Legendre transform of $F_0$. One has that  $G_0=+\infty$ in $\R^n \setminus P$ and, for
$s \in \inte P=\nabla F_0(\R^n)$,
$$G_0(s)=\langle x,s\rangle-F_0(x)\Longleftrightarrow s\in\nabla F_0(x)\Longleftrightarrow x\in\nabla G_0(s)\,.$$

Guillemin observed in \cite{Gui94} that a ``natural" representative of the cohomology class $\{\omega_P\}$
is given by
$$
G(s)=\frac{1}{2} \,\sum_{i=1}^d \ell_i(s) \log \ell_i(s).
$$
We refer the reader to \cite{CDG02} for a neat proof of this beautiful formula of Guillemin.
Observe that $G$ is only Log-Lipschitz regular on $P$, although the original K\"ahler potential
is smooth.

The purpose of this note is to undertake a systematic study of toric pluripotential analysis.
There are three ways to understand a toric quasiplurisubharmonic (qpsh) function and its Monge-Amp\`ere measure:
\begin{itemize}
\item by working directly on $X$ and imposing toric symmetries,
\item by looking at the corresponding object (convex function, real Monge-Amp\`ere measure)  in $\R^n$
after a logarithmic transformation, and understanding the asymptotic properties at infinity,
\item by understanding the behavior near the boundary of the polytope
of the Legendre transform of the corresponding convex function.
\end{itemize}

\smallskip

We refer to Section \ref{S:toricenergy} for the definition of toric $\omega$-plurisubharmonic ($\omega$-psh) functions on $X$ and the corresponding energy classes. If $\varphi$ is $\omega$-psh, we denote by $F_\varphi$ the corresponding convex function on $\R^n$ and by $G_\varphi$ its Legendre transform (see Sections \ref{S:convex} and \ref{S:toricenergy}).

Our main results are as follows. We first describe  the class of toric $\omega$-psh functions (see Proposition \ref{P:torpsh}):

\medskip

\noindent
{\bf Proposition A.}
{\it
Let $F_P(x)=\max_{s \in P} \langle x,s \rangle$ denote the support function of the polytope $P$.
The following are equivalent:

(i) $\varphi\in PSH_{tor}(X,\omega)$;

(ii) $F_\varphi\leq F_P+C$ for some constant $C$;

(iii) $G_\varphi=+\infty$ on $\R^n\setminus P$;

(iv) $\nabla F_\varphi(\R^n)\subset P$.

}

\medskip

We then characterize finite energy toric $\omega$-psh functions and their weighted versions, showing in particular the following (see Theorem \ref{T:Etor}):

\medskip

\noindent
{\bf Theorem B.}
{\it
Let $\f \in PSH_{tor}(X,\omega)$. The following  are equivalent:

(i) $\f \in \E_{tor}(X,\omega)$;

(ii) $G_\f$ is finite on $\inte P$;

(iii) $F_\f$ has full Monge-Amp\`ere mass;

(iv) the Lelong numbers $\nu(\varphi,p)=0$ for all $p\in X$. }

\medskip

In Theorem \ref{thm:holder} we study more regular toric $\omega$-psh functions,
characterizing the maximal Log-Lipschitz regularity of Legendrian potentials:

\medskip

\noindent
{\bf Theorem C.}
{\it
Let $\f \in \E_{tor}(X,\omega)$. The following properties are equivalent:

(i) There exists $\e>0$ such that $\exp(-\e PSH_{tor}(X,\omega) ) \subset L^1(MA(\f))$;

\smallskip

(ii)  The function $G_\f$ is Log-Lipschitz on $P$.}

\medskip

It is tempting to think that these conditions are all equivalent to the fact that
$\f$ is H\"older continuous. This is easily seen to be the case when $n=1$.
We refer the interested reader to \cite{DDGHKZ} for more information, geometric motivations,
 and related questions connecting the
H\"older continuity of Monge-Amp\`ere potentials to the integrability properties of the associated
complex Monge-Amp\`ere measure.

The paper is organized as follows. In Section \ref{S:eclass} we recall some basic facts about $\omega$-psh functions on any compact K\"ahler manifold $(X,\omega)$, together with the definition and main properties of various energy classes following \cite{GZ07}. Section \ref{S:convex} deals with the relevant properties of convex functions and their Legendre transforms. In Section \ref{S:toricenergy} we study energy classes of toric $\omega$-psh functions on a toric compact K\"ahler manifold $(X,\omega)$, and in Section \ref{S:LogLip} we conclude by looking at questions about the higher regularity of such functions.

\begin{ackn}
This article has been written during the postdoctoral research period of the third named author at l'Institut de Math\'{e}matiques de Toulouse. She is grateful to her co-authors for their endless support and hospitality during the stay.
\end{ackn}

\section{Finite energy classes}\label{S:eclass}

In this section we let $(X,\omega)$ be a compact K\"ahler manifold of dimension $n$, and we recall the definition of finite energy classes of quasiplurisubharmonic (qpsh) functions following \cite{GZ07}.

 \subsection{Bedford-Taylor theory}
A function on $X$ is qpsh if it is locally the sum of a psh function and a smooth one. In particular qpsh functions are upper semicontinuous and integrable.

\begin{Definition}
A function $\f:X \rightarrow \R \cup \{-\infty\}$ is $\omega$-plurisubharmonic ($\omega$-psh) if it is qpsh and if the current $\omega+dd^c \f$ is positive on $X$.
\end{Definition}

Let $PSH(X,\omega)$ denote the set of all $\omega$-psh functions on $X$. This is a closed subset of $L^1(X,\omega^n)$.

Bedford and Taylor showed in \cite{BT82} that one can define the complex Monge-Amp\`ere operator
$$
MA(\f):=(\omega+dd^c \f)^n=(\omega+dd^c \f)\wedge\ldots\wedge(\omega+dd^c \f)
$$
for all {\it bounded} $\omega$-psh functions. They showed that whenever $(\f_j)$ is a sequence of bounded
$\omega$-psh functions decreasing locally to $\f$, the sequence of measures $MA(\f_j)$
converges weakly towards the measure $MA(\f)$. Note also that
$$\int_XMA(\f)=\int_X\omega^n=:V_\omega.$$

At the heart of Bedford-Taylor's theory lies the following {\it maximum principle}: if $u,v$ are bounded
$\omega$-psh functions, then
$$
\hskip-3cm (MP) \hskip2cm 1_{\{v<u\}} MA(\max(u,v)) =1_{\{v<u\}} MA(u).
$$

The maximum principle $(MP)$ implies the so called {\it comparison principle}:
if  $u,v$ are bounded $\omega$-psh functions then
$$
\int_{\{v<u\}}  MA(u) \leq \int_{\{v<u\}}  MA(v).
$$

\subsection{The class ${\mathcal E}(X,\omega)$}\label{SS:E}

If   $\f \in PSH(X,\omega)$, we let
$$
\f_j:=\max(\f, -j) \in PSH(X,\omega) \cap L^{\infty}(X).
$$
It follows from the Bedford-Taylor theory that the measures $MA(\f_j)$ are well defined measures of total mass $V_\omega$. The following monotonicity property holds:
$$
\mu_j:={\bf 1}_{\{ \f>-j\}} MA(\f_j)
\text{ is an increasing sequence of Borel measures.}
$$

\noindent The proof is an elementary consequence of $(MP)$ (see \cite[p.445]{GZ07}). Since $\mu_j$ have total mass bounded above by $V_\omega$, we can define
$$
\mu_{\f}:=\lim_{j \rightarrow +\infty} \mu_j,
$$
which is a positive Borel measure on $X$ of total mass $\leq V_\omega$.

\begin{Definition}
We let
$$
{\mathcal E}(X,\omega):=\left\{ \f \in PSH(X,\omega):\, \mu_{\f}(X)=V_\omega \right\}.
$$
For $\f \in {\mathcal E}(X,\omega)$, we set
$
MA(\f):=\mu_{\f}.
$
\end{Definition}

The definition is justified by the following important fact proved in \cite{GZ07}:
{\it the complex Monge-Amp\`ere operator $\f \mapsto MA(\f)$ is well defined on the class
${\mathcal E}(X,\omega)$, in the sense that if $\f \in {\mathcal E}(X,\omega)$ then for every decreasing sequence of bounded  $\omega$-psh functions $\f_j\searrow\f$, the  measures $MA(\f_j)$ converge weakly on $X$ towards $\mu_\f$.}

Every bounded $\omega$-psh function clearly belongs to ${\mathcal E}(X,\omega)$.
The class ${\mathcal E}(X,\omega)$ also contains many $\omega$-psh functions which are unbounded.
When $X$ is a compact Riemann surface, ${\mathcal E}(X,\om)$ is the set of $\om$-sh functions
whose Laplacian does not charge polar sets.

\begin{Remark}
If $\f \in PSH(X,\omega)$ is normalized so that $\f \leq -1$, then $-(-\f)^\e $ belongs to
${\mathcal E}(X,\om)$ whenever $0 \leq \e <1$ (see e.g.\ \cite{CGZ08}).
The functions which belong to the class ${\mathcal E}(X,\om)$,
although usually unbounded,
have relatively mild singularities. In particular they have zero Lelong number at every point.
\end{Remark}

It is shown in \cite{GZ07} that the maximum principle $(MP)$ and the comparison principle
continue to hold in the class ${\mathcal E}(X,\omega)$. The latter can be characterized as the largest class for which the complex Monge-Amp\`ere operator is well defined and the maximum principle holds.

\subsection{Weighted energy classes}\label{SS:Echi}

Let $\W$ denote the set of all functions $\chi:\R^- \rightarrow \R^-$ such that $\chi$ is increasing and $\chi(-\infty)=-\infty$.

\begin{Definition}
We let
${\mathcal E}_{\chi}(X,\om)$ be the set of
$\om$-psh
functions with
finite $\chi$-energy,
$$
{\mathcal E}_{\chi}(X,\om):=\left\{ \f \in {\mathcal E}(X,\om):\,\chi(-|\f|) \in L^1(X, MA(\f)) \right\}.
$$
When $\chi(t)=-(-t)^p$, $p>0$, we set $\E^p(X,\omega)=\E_\chi(X,\omega)$.
\end{Definition}

We list here a few important properties of these classes and refer the reader to \cite{GZ07,BEGZ10} for the proofs:
\begin{itemize}
\item $\E(X,\omega)=\bigcup_{\chi \in \W} \, \E_\chi(X,\omega)$;
\item $PSH(X,\om) \cap L^{\infty}(X)=\bigcap_{\chi \in {\mathcal W}} \E_\chi(X,\omega)$;
\item the classes $\E^p(X,\omega)$ are  convex;
\item $\f \in \E^p(X,\omega)$ if and only if for any (resp.\ for one)  sequence of bounded
$\omega$-psh functions $\f_j\searrow\f$,
$\sup_j \int_X |\f_j|^p MA(\f_j) <+\infty$.
\item if  $\f_j$ is a sequence of  $\omega$-psh functions decreasing to $\f \in \E^p(X,\omega)$, then
the measures $|\f_j|^p MA(\f_j)$ converge weakly to $|\f|^p MA(\f)$.
\end{itemize}

\section{Facts on convex functions}\label{S:convex}

We collect here a few properties of convex functions which will be used later. Some of these are well known and proofs are included for the convenience of the reader (see also \cite[Section 2]{BerBer13}).

\subsection{Subgradients and  Monge-Amp\`ere measures}

Let $F:{\R}^n\to\R$ be a convex function. The subgradient of $F$ at $x$ is the set
$$\nabla F(x)=\{s\in\R^n:\,F(y)\geq F(x)+\langle y-x,s\rangle,\;\forall\,y\in\R^n\}.$$
We let
$$
\nabla F(\R^n):=\bigcup_{x\in\R^n}\nabla F(x)\,.
$$

The Legendre transform $G$ of $F$ is the lower semicontinuous convex function defined by
$$G:\R^n\to(-\infty,+\infty]\,,\,\;G(s)=\sup_{x\in\R^n}(\langle x,s\rangle-F(x))\,.$$
Then $F$ is the Legendre transform of $G$,
$$F(x)=\sup_{s\in\R^n}(\langle x,s\rangle-G(s))\,,$$
and one has
$$G(s)=\langle x,s\rangle-F(x)\Longleftrightarrow s\in\nabla F(x)\Longleftrightarrow x\in\nabla G(s)\,.$$

\begin{Lemma}\label{L:conv1} Let $F:{\R}^n\to\R$ be a convex function.

(i) If $F$ is smooth and strictly convex then $\nabla F:\R^n\to\R^n$ is injective, and hence an open map.

(ii) If $s_0\in\nabla F(\R^n)$ then $G(s_0)<+\infty$. Conversely, if $G(s)<+\infty$ for all $s$ in an open ball $B(s_0,r)$ then $s_0\in\nabla F(\R^n)$.

(iii) Let $F_j:{\R}^n\to\R$, $j\geq1$, be convex functions. Then $F_j\searrow F$ pointwise on $\R^n$ if and only if the Legendre transforms $G_j\nearrow G$ pointwise on $\R^n$.
\end{Lemma}

\begin{proof} $(i)$ If $p\neq q$ and $f(t):=F((1-t)p+tq)$ then $f''(t)>0$, so
$f'(0)=\langle\nabla F(p),q-p\rangle<f'(1)=\langle\nabla F(q),q-p\rangle$. Hence $\nabla F(p)\neq\nabla F(q)$.

$(ii)$ By the definition of the subgradient, if $s_0\in\nabla F(x)$ then  $\langle y,s_0\rangle-F(y)\leq \langle x,s_0\rangle-F(x)$ for all $y\in\R^n$, so $G(s_0)=\langle x,s_0\rangle-F(x)<+\infty$. Conversely, by shrinking $r$ we may assume that $G<M$ on $B(s_0,r)$ for some constant $M$, hence $\langle x,s\rangle-F(x)\leq M$ for all $x\in\R^n$ and $s\in B(s_0,r)$. Let $\wi F(x)=F(x)-\langle x,s_0\rangle+M$. It follows that $\wi F(x)\geq\langle x,s-s_0\rangle$ for all $s\in B(s_0,r)$, hence $\wi F(x)\geq r\|x\|$. Therefore $\wi F$ assumes a global minimum, i.e. there exists $x_0\in\R^n$ such that $\wi F(x)\geq\wi F(x_0)$. Thus $0\in\nabla\wi F(x_0)=\nabla F(x_0)-s_0$. (Note that if $F(x)=e^x$, $x\in\R$, then $G(0)=0$ but $0\not\in F'(\R)$, so the hypothesis that $G(s)<+\infty$ in a neighborhood of $s_0$ is needed.)

$(iii)$ Assume that $F_j\searrow F$. Then $G_j\nearrow\wi G$, where $\wi G$ is lower semicontinuous, convex and $\wi G\leq G$. If $\wi F$ is the Legendre transform of $\wi G$ we have that $F_j\geq\wi F\geq F$. We conclude that $\wi F=F$ and so $\wi G=G$. The converse follows by a similar argument.
\end{proof}

\begin{Lemma}\label{L:MAR} Let $F:{\R}^n\to\R$ be a convex function. If $\chi$ is a continuous function with compact support on $\R^n$ then
$$\int_{(\C^\star)^n}(\chi\circ L)\,(dd^c F\circ L)^n=\int_{\R^n} \chi\,MA_\R(F)\,,$$
where $L$ is defined in \eqref{e:L}, $d=\partial+\db$, $d^c=\frac{1}{2\pi i}\,(\partial-\db)$, and $MA_\R(F)$ is the real Monge-Amp\`ere measure of $F$.
\end{Lemma}

\begin{proof} Approximating $F$ by a decreasing sequence of smooth convex functions it suffices to assume that $F$ is smooth. Recall that in this case $MA_\R(F)$ is the measure defined by
$$MA_\R(F)=n!\,\det\left[\frac{\partial^2F}{\partial x_i\partial x_j}\right]dV\,,$$
where $V$ denotes the Lebesgue measure on the corresponding Euclidean space. Note that the function $F\circ L$ is psh on $(\C^\star)^n$ and
$$\frac{\partial^2 (F\circ L)}{\partial z_i \partial \overline{z_j}}=\frac{1}{4z_i\ov z_j} \left(\frac{\partial^2 F}{\partial x_i \partial x_j} \circ L\right),$$
hence
$$\det \left[ \frac{\partial^2 (F\circ L)}{\partial z_i \partial \overline{z_j}} \right]=\frac{1}{4^n\prod_j |z_j|^2} \left(\det \left[\frac{\partial^2 F}{\partial x_i \partial x_j} \right] \circ L\right).$$
It follows that
\begin{align*}
(dd^c F\circ L)^n&=\left(\frac{i}{\pi}\right)^n\big(\partial\db F\circ L\big)^n\\
&=n!\left(\frac{i}{\pi}\right)^n\det \left[ \frac{\partial^2 (F\circ L)}{\partial z_i \partial \overline{z_j}} \right]\,dz_1\wedge d\ov z_1\wedge\ldots\wedge dz_n\wedge d\ov z_n\\
&=n!\left(\frac{2}{\pi}\right)^n\det \left[ \frac{\partial^2 (F\circ L)}{\partial z_i \partial \overline{z_j}} \right]\,dV(z)\\
&=n!\left(\frac{2}{\pi}\right)^n\frac{1}{4^n\prod_j |z_j|^2} \left(\det \left[\frac{\partial^2 F}{\partial x_i \partial x_j} \right] \circ L\right)dV(z)\\
&=\frac{n!}{(2\pi)^n}\,\det\left[\frac{\partial^2 F}{\partial x_i \partial x_j}\,(\log r_1,\ldots,\log r_n) \right]\frac{dr_1\ldots dr_n}{r_1\ldots r_n}\,d\theta_1\ldots d\theta_n,
\end{align*}
where we used polar coordinates $z_j=r_je^{i\theta_j}$. Changing variables $x_j:=\log r_j$ we obtain
\begin{eqnarray*}
\lefteqn{
\int_{(\C^\star)^n}(\chi\circ L)\,(dd^c F\circ L)^n=} \\
&=n!\int_{(0,+\infty)^n}\big(\chi\det\left[\frac{\partial^2 F}{\partial x_i \partial x_j}\right]\big)(\log r_1,\ldots,\log r_n)\,\frac{dr_1\ldots dr_n}{r_1\ldots r_n}\\
&=n!\int_{\R^n}\chi(x)\det\left[\frac{\partial^2 F}{\partial x_i \partial x_j}\,(x)\right]dV(x)=\int_{\R^n}\chi\,MA_\R(F)\,.
\end{eqnarray*}

For a non-smooth convex function $F$ the positive measure $MA_\R(F)$ is the real Monge-Amp\`ere measure of $F$ (in the sense of Alexandrov \cite{Gut01}).
\end{proof}

The following lemma is proved using an idea of Al Taylor \cite{T82}.

\begin{Lemma}\label{L:AT} If $F_1,F_2:\R^n\to\R$ are convex functions such that $F_2(x)\to+\infty$ as $\|x\|\to+\infty$ and $F_1(x)\leq F_2(x)$ for all $x\in\R^n$, then
$$\int_{\R^n}MA_\R(F_1)\leq\int_{\R^n}MA_\R(F_2)\,.$$
\end{Lemma}

\begin{proof} Fix a compact $K\subset(\C^\star)^n$, a number $\varepsilon>0$, and consider the psh function on $(\C^\star)^n$,
$$u:=\max\{F_1\circ L,(1+\varepsilon)F_2\circ L-C\}\,,$$
where the constant $C>0$ is chosen such that $u=F_1\circ L$ in a neighborhood of $K$. Since $F_2(x)\to+\infty$ as $\|x\|\to+\infty$ it follows that $F_1\leq F_2\leq(1+\varepsilon)F_2-C$ on $\R^n\setminus\mathcal K$ for some compact $\mathcal K\subset\R^n$. Then $L^{-1}(\mathcal K)\subset(\C^\star)^n$ is compact and $u=(1+\varepsilon)F_2\circ L-C$ on $(\C^\star)^n\setminus L^{-1}(\mathcal K)$. We infer that
$$(1+\varepsilon)^n\int_{(\C^\star)^n}(dd^cF_2\circ L)^n=\int_{(\C^\star)^n}(dd^cu)^n\geq\int_K(dd^cF_1\circ L)^n\,.$$
The lemma follows by using Lemma \ref{L:MAR} and by letting $K\nearrow(\C^\star)^n$ and $\varepsilon\searrow0$.
\end{proof}

\subsection{Growth properties}

Let $P$ be a (compact) convex body in $\R^{n}$. Its {\em support function}, which is also known as the \emph{indicator function}, is the convex function
$$F_P(x):=\max_{s \in P} \langle x,s \rangle.$$
Its Legendre transform is the convex function
$$
G_P(x)=\left\{\begin{array}{ll}0, \,\text{ \ \ \ \ if } x\in P,\\+\infty, \text{ if } x\not\in P.\end{array}\right.
$$

If $P_\vartheta=\theta+P$ is the image of $P$ under the translation by $\vartheta$, and $P_\lambda=\lambda P$ is the image of $P$ under the dilation by $\lambda >0$, then
\begin{align*}
&F_{P_\vartheta}(x)=F_P(x)+<\vartheta,x> \,,\,\;G_{P_\vartheta}(s)=G_P(s-\vartheta),\\
&F_{P_\lambda}(x)=F_P(\lambda x)=\lambda F_P(x)\,,\,\;G_{P_\lambda}(s)=G_P\big(\frac{s}{\lambda}\big).
\end{align*}

\begin{Lemma}\label{L:conv2}
Let $F:{\R}^n\to\R$ be a convex function with Legendre transform $G$. The following are equivalent:

(i) $F\leq F_P+C$ for some constant $C$;

(ii) $G=+\infty$ on $\R^n\setminus P$;

(iii) $\nabla F(\R^n)\subset P$.
\end{Lemma}

\begin{proof}
To show that $(i)\Rightarrow(ii)$, if $F\leq F_P+C$ then $G\geq G_P-C$, so $G=G_P=+\infty$ on $\R^n\setminus P$. For $(ii)\Rightarrow(iii)$, if $s\in\nabla F(\R^n)$ then $G(s)<+\infty$, hence $s\in P$ by $(ii)$.

To prove that $(iii)\Rightarrow(i)$, let $x\in\R^n$ and note that if $s\in\nabla F(\R^n)$ then $s\in P$, so $\langle s,x\rangle\leq F_P(x)$. Since $F$ is locally Lipschitz along the line $t\in\R\to tx$ we have
$$
F(x)-F(0)=\int_0^1\frac{d}{dt}\,F(tx)\,dt=\int_{0}^{1}\langle\nabla F(tx),x\rangle\,dt\leq\int_{0}^{1}F_P(x)\,dt=F_P(x).
$$
 \end{proof}

\begin{Lemma}\label{L:conv3}
Let $F_0:{\R}^n\to\R$ be a smooth strictly convex function such that $F_P-C\leq F_0\leq F_P+C$ for some constant $C$. Then $\nabla F_0:\R^n\to\inte P$ is bijective and $\nabla G_0:\inte P\to \R^n$ is its inverse, where $G_0$ is the Legendre transform of $F_0$. Moreover, if $\chi$ is a continuous function with compact support on $\R^n$ then
$$\int_{\R^n} \chi\,MA_\R(F_0)=n!\int_{\inte P}\chi\circ\nabla G_0\;dV\,,\;\text{ and }\,\int_{\R^n}\,MA_\R(F_0)=n!\vol(P)\,.$$
\end{Lemma}

\begin{proof}
By Lemma \ref{L:conv2} and Lemma \ref{L:conv1} $(i)$, $\nabla F_0:\R^n\to P$ is injective. As $F_P-C\leq F_0$ we have that $G_0\leq G_P+C$, so $G_0\leq C$ on $P$. Thus $\inte P\subset\nabla F_0(\R^n)$ by Lemma \ref{L:conv1} $(ii)$, and hence $\nabla F_0(\R^n)=\inte P$ since $\nabla F_0$ is open. If $x,x'\in\nabla G_0(s)$ then $s=\nabla F_0(x)=\nabla F_0(x')$, so $x=x'$. Hence $G_0$ is differentiable on $\inte P$ and $\nabla G_0=(\nabla F_0)^{-1}$. The remaining assertions of the lemma follow by the change of variables $x=\nabla G_0(s)$, $s=\nabla F_0(x)$, so
$$
dV(s)=\det\left[\frac{\partial^2F_0}{\partial x_i\partial x_j}\right]dV(x)=\frac{1}{n!}\,MA_\R(F_0)(x)\,.
$$
\end{proof}

\begin{Lemma}\label{L:conv4}
If $0\in\inte P$ then there exist constants $a,b >0$ such that
$$b\|x\| \leq F_P(x)  \leq a\|x\|,\;\forall\,x\in\R^n.$$
\end{Lemma}

\begin{proof}
If $a,b>0$ are such that the closed balls $\ov B(0,b)\subset P \subset\ov B(0,a)$, then
$$
b\|x\|=F_{\ov B(0,b)} \leq F_P(x) \leq F_{\ov B(0,a)}=a\|x\|.
$$
\end{proof}

\begin{Lemma}\label{L:conv5}
Assume that $0\in\inte P$ and let $F:{\R}^n\to\R$ be a convex function with Legendre transform $G$, such that $F\leq F_P+C$ for some constant $C$. The following are equivalent:

(i) $G(s)<+\infty$ for all $s\in\inte P$;

(ii) for every $\varepsilon\in(0,1)$ there is $M_\varepsilon>0$ s.t. $F\geq(1-\varepsilon)F_P-M_\varepsilon$ on $\R^n$.

\noindent Moreover, these conditions imply that $\int_{\R^n}\,MA_\R(F)=n!\vol(P)$.
\end{Lemma}

\begin{proof} Note that
$$F_{(1-\varepsilon)P}(x)=\sup_{s\in P}\langle x,(1-\varepsilon)s\rangle=(1-\varepsilon)F_P(x)\,.$$

Assume that $G(s)<+\infty$ for all $s\in\inte P$. Since $0\in\inte P$, $(1-\varepsilon)P\subset\inte P$ for $\varepsilon\in(0,1)$, so there exists $M_\varepsilon>0$ such that $G\leq M_\varepsilon$ on $(1-\varepsilon)P$. It follows that
$$F(x)\geq\sup_{s\in(1-\varepsilon)P}\big(\langle x,s\rangle-G(s)\big)\geq F_{(1-\varepsilon)P}(x)-M_\varepsilon =(1-\varepsilon)F_P(x)-M_\varepsilon\,.$$
Conversely, if $F\geq(1-\varepsilon)F_P-M_\varepsilon=F_{(1-\varepsilon)P}-M_\varepsilon$, then $G\leq G_{(1-\varepsilon)P}+M_\varepsilon$, so $G(s)\leq M_\varepsilon$ for $s\in(1-\varepsilon)P$. As $\varepsilon\searrow 0$ this implies that $G(s)<+\infty$ for all $s\in\inte P$.

By Lemma \ref{L:conv4} we have that $F_P(x)\to+\infty$ as $\|x\|\to+\infty$. Since $F\leq F_P+C$, Lemmas \ref{L:AT} and \ref{L:conv3} imply that
$$\int_{\R^n}\,MA_\R(F)\leq\int_{\R^n}\,MA_\R(F_P)=\int_{\R^n}\,MA_\R(F_0)=n!\vol(P)\,,$$
where $F_0$ is a function as in Lemma \ref{L:conv3}. Note that by $(ii)$, $F(x)\to+\infty$ as $\|x\|\to+\infty$, hence Lemma \ref{L:AT} again shows that
$$\int_{\R^n}\,MA_\R(F)\geq(1-\varepsilon)^n\int_{\R^n}\,MA_\R(F_P)\,,\,\;\forall\,\varepsilon\in(0,1)\,.$$
Letting $\varepsilon\to0$ finishes the proof.
\end{proof}

We conclude this section with the following lemma:

\begin{Lemma}\label{L:conv6} Let $P$ be a compact convex body in $\R^n$ with nonempty interior and $G:P\to\R\cup\{+\infty\}$ be a lower semicontinuous convex function. Then
$$\big|\inf_PG\big|\leq\frac{1}{\big(2^{1/(n+1)}-1\big)\vol(P)}\,\int_P|G|\,dV\,.$$
\end{Lemma}

\begin{proof}
 If $G\geq0$ on $P$ then $\int_PG\,dV\geq\inf_PG\cdot\vol(P)$ and we are done. Otherwise, consider the convex set $S=\{G<0\}\subset P$. It suffices to show that if $p\in\inte S$ then
$$-G(p)\leq\frac{1}{\big(2^{1/(n+1)}-1\big)\vol(P)}\,\int_P|G|\,dV\,.$$
We assume without loss of generality that $p=0$ and use spherical coordinates. For $\theta\in S^{n-1}$ let $0<a(\theta)\leq b(\theta)$ be defined by $a(\theta)\theta\in\partial S$, $b(\theta)\theta\in\partial P$. If $\sigma$ is the area measure on $S^{n-1}$ we have that
$$\vol(P)=\int_{S^{n-1}}\frac{b(\theta)^n}{n}\,d\sigma(\theta)\,.$$
By convexity it follows that
\begin{align*}
G(t\theta)&\leq\frac{-G(0)}{a(\theta)}\,(t-a(\theta))\leq0\,,\,\text{ if } 0\leq t\leq a(\theta)\,,\\
G(t\theta)&\geq\frac{-G(0)}{a(\theta)}\,(t-a(\theta))\geq0\,,\,\text{ if } a(\theta)<t\leq b(\theta)\,.
\end{align*}
Thus
\begin{align*}
\int_P|G|\,dV\;\geq\;&\int_{S^{n-1}}\int_0^{a(\theta)}\frac{-G(0)}{a(\theta)}\,(a(\theta)-t)t^{n-1}dt\,d\sigma(\theta)\;+\\
&\;\int_{S^{n-1}}\int_{a(\theta)}^{b(\theta)}\frac{-G(0)}{a(\theta)}\,(t-a(\theta))t^{n-1}dt\,d\sigma(\theta)\\
=&\;\frac{-G(0)}{n(n+1)}\int_{S^{n-1}}\left(\frac{nb(\theta)^{n+1}}{a(\theta)}-(n+1)b(\theta)^n+2a(\theta)^n\right)d\sigma(\theta)\,.
\end{align*}
Note that
$$
f(a):=\frac{nb^{n+1}}{a}-(n+1)b^n+2a^n\geq f\big(b2^{-\frac{1}{n+1}}\big)=(n+1)\big(2^{\frac{1}{n+1}}-1\big)b^n,
$$
$\text{ for } 0<a\leq b.$
Therefore
\begin{equation*}
\begin{split}\int_P|G|\,dV&\geq\frac{-G(0)}{n}\,\big(2^{\frac{1}{n+1}}-1\big)\int_{S^{n-1}}b(\theta)^n\,d\sigma(\theta)\\
&=-G(0)\big(2^{1/(n+1)}-1\big)\vol(P)\,,
\end{split}
\end{equation*}
and we are done.
\end{proof}

\section{Toric energy classes}\label{S:toricenergy}

Let $(X,\omega)$ be a toric compact K\"ahler manifold of dimension $n$. Then $X$ is a compactification of the complex torus $(\C^\star)^n$ such that the canonical action by multiplication of $(\C^\star)^n$ on itself extends to a holomorphic action of $(\C^\star)^n$ on $X$. Moreover, there exists a smooth strictly convex function $F_0:\R^n\to\R$ such that $\omega\mid_{(\C^\star)^n}=dd^cF_0\circ L$, where $L$ is defined in \eqref{e:L}. If $P$ is the compact convex polytope determined by $X$ then $\nabla F_0:\R^n\to\inte P$ is bijective and we may assume that $0\in\inte P$. Let $G_0$ denote the Legendre transform of $F_0$.

\subsection{Toric qpsh functions}

A toric $\omega$-psh function on $X$ is an $\omega$-psh function $\varphi$ that is invariant under the $(S^1)^n$ action induced by the $(\C^\star)^n$ action on $X$. We denote by $PSH_{tor}(X,\omega)$ the class of such functions. It follows that there exists a convex function $F_\varphi:\R^n\to\R$ such that
$$F_\varphi\circ L=F_0\circ L+\varphi\,\text{ on }(\C^\star)^n\subset X\,.$$
We denote by $G_\varphi$ the Legendre transform on $F_\varphi$. Note that $F_\varphi$ is continuous on $\R^n$, hence $\varphi$ is continuous on $(\C^\star)^n$.

We define the energy classes of toric $\omega$-psh functions by
\begin{equation*}
\begin{split}
{\mathcal E}_{tor}(X,\omega)&=PSH_{tor}(X,\omega)\cap{\mathcal E}(X,\omega),\\
{\mathcal E}^p_{tor}(X,\omega)&=PSH_{tor}(X,\omega)\cap{\mathcal E}^p(X,\omega),\\
{\mathcal E}_{\chi,tor}(X,\omega)&=PSH_{tor}(X,\omega)\cap{\mathcal E}_\chi(X,\omega),
\end{split}
\end{equation*}
where $p>0$ and $\chi\in\W$ (see sections \ref{SS:E}, \ref{SS:Echi}).

We begin with the following simple lemma:

\begin{Lemma}\label{L:F0FP} There exists a constant $C>0$ such that
$$-C\leq F_0(x)-F_P(x)\leq C\,,\;\forall\,x\in\R^n.$$

\end{Lemma}

\begin{proof} Since $\nabla F_0(\R^n)\subset P$ we have by Lemma \ref{L:conv2} that $F_0\leq F_P+C_1$ for some constant $C_1$. Let $\omega'\in\{\omega\}$ be a K\"ahler form with associated convex function $F$ such that its Legendre transform $G$ is given by Guillemin's formula. Then $F\circ L=F_0\circ L+\theta$ for some smooth $\omega$-psh function $\theta$. Hence $F\leq F_0+C_2$ and $G\geq G_0-C_2$, for some constant $C_2$. Since $G$ is bounded above on $P$ it follows that $G_0\leq G_P+C_3$, and so $F_0\geq F_P-C_3$, for some constant $C_3$.
\end{proof}

Our next result gives a characterization of toric $\omega$-psh functions:

\begin{Proposition}\label{P:torpsh} The following are equivalent:

(i) $\varphi\in PSH_{tor}(X,\omega)$;

(ii) $F_\varphi\leq F_P+C$ for some constant $C$;

(iii) $G_\varphi=+\infty$ on $\R^n\setminus P$;

(iv) $\nabla F_\varphi(\R^n)\subset P$.
\end{Proposition}

\begin{proof}
If $\varphi\in PSH_{tor}(X,\omega)$ then $\varphi$ is bounded above on $X$, hence $F_\varphi\leq F_0+C'$ for some constant $C'$, and $(ii)$ follows by Lemma \ref{L:F0FP}.

Conversely, if $(ii)$ holds then by Lemma \ref{L:F0FP}, $F_\varphi\leq F_0+C'$ for some constant $C'$, hence $\varphi\leq C'$ on $(\C^\star)^n\subset X$. Since $X\setminus(\C^\star)^n$ is an analytic set invariant under the $(S^1)^n$ action, we conclude that $\varphi$ extends to an $\omega$-psh function on $X$ which is $(S^1)^n$ invariant.

The remaining equivalences $(ii)\Leftrightarrow(iii)\Leftrightarrow(iv)$ follow from Lemma \ref{L:conv2}.
\end{proof}

\begin{Proposition} \label{P:supest}
 If $\varphi\in PSH_{tor}(X,\omega)$ then
$$
\sup_X\varphi\leq C_P+\frac{1}{\big(2^{1/(n+1)}-1\big)\vol(P)}\,\int_P|G_\varphi|\,dV\,,
$$
where $C_P=\sup_PG_0=\sup_{\R^n}(F_P-F_0)$.
\end{Proposition}

\begin{proof}
Note that, for a constant $C$, one has $F_\varphi-F_0\leq C$ on $\R^n$ if and only if $G_0-G_\varphi\leq C$ on $P$. It follows that
$$
\sup_X\varphi=\sup_{\R^n}(F_\varphi-F_0)=\sup_P(G_0-G_\varphi)\leq C_P-\inf_PG_\varphi\,,
$$
and the proposition follows from Lemma \ref{L:conv6}.
\end{proof}

\begin{Example}
Let $X=\F_1$ be
the blow up of $\P^2$ at a toric point $p$.
It is a geometrically ruled surface.
We let $F$ denote a generic fiber, $E$ be the exceptional divisor, and $H=E+F$ the total transform of a line through $p$. The cohomology classes of $F$ and $H$ are both semi-positive and
generate $H^{1,1}(X,\R)$. Any K\"ahler class $\{\omega\}$ is cohomologous to $aH+bF$, with $a,b>0$.
In coordinates $z \in (\C^*)^2$ it can be represented by
$$\omega=a \omega_1+b\omega_2, \text{ where } \omega_1=\frac{1}{2}\,dd^c \log(1+\|z\|^2),\;\omega_2=dd^c \log\|z\|.$$
The convex function associated to $\omega$ is
$$
F_0(x)=\frac{a}{2}\,\log \left( 1+e^{2x_1}+e^{2x_2} \right)+\frac{b}{2}\,\log \left(e^{2x_1}+e^{2x_2} \right),
$$
and $P=\overline{\nabla F_0(\R^2)}$ is the polytope
$$
P=\left\{s_1\geq0,\;s_2 \geq 0, \, b\leq s_1+s_2 \leq a+b\right\}.
$$
Thus $d=4$, $\ell_1(s)=s_1$, $\ell_2(s)=s_2$, $\ell_3(s)=a+b-s_1-s_2$, and $\ell_4(s)=s_1+s_2-b$.
For $s\in P$, the Legendre transform of $F_0$ is given by
\begin{equation*}\begin{split}
G_0(s) &= \frac{1}{2} \big[ s_1 \log s_1 + s_2 \log s_2 +(a+b-s_1-s_2) \log (a+b-s_1-s_2)+  \\
& (s_1+s_2-b)\log (s_1+s_2-b)-(s_1+s_2)\log(s_1+s_2)-a\log a \big].
\end{split}\end{equation*}
\end{Example}

\subsection{The class ${\mathcal E}_{tor}(X,\omega)$}

\begin{Definition}\label{D:fullMA}
Let $F:\mathbb{R}^n\rightarrow\mathbb{R}$ be a convex function such that $F\leq F_P+C$. We say that $F$ has \emph{full Monge-Amp\`{e}re mass  } if
$$\int_{\R^n}\,MA_\R(F)=\int_{\R^n}\,MA_\R(F_0)=n!\vol(P)\,.$$
\end{Definition}

Recall that a {\em toric point} of $X$ is a point fixed by the action of the complex torus $(\C^\star)^n$ on $X$.

\begin{Theorem} \label{T:Etor} Let $\f \in PSH_{tor}(X,\omega)$. The following  are equivalent:

(i) $\f \in \E_{tor}(X,\omega)$;

(ii) $G_\f$ is finite on $\inte P$;

(iii) $F_\f$ has full Monge-Amp\`ere mass;

(iv) for every $\varepsilon>0$ there exists a compact set $K_\varepsilon\subset(\C^\star)^n$  such that
$$\varphi(z)\geq-\varepsilon\max\big\{\big|\log|z_1|\big|,\ldots,\big|\log|z_n|\big|\big\}\,\text{ on } (\C^\star)^n\setminus K_\varepsilon.$$

(v) the Lelong numbers $\nu(\varphi,p)=0$ for all $p\in X$.

(vi) the Lelong numbers $\nu(\varphi,p)=0$ at all toric points $p\in X$.
  \end{Theorem}

\begin{proof}
To prove that $(i)\Rightarrow (ii)$, if $\f \in \E_{tor}(X,\omega)$ then $\f \in \E_{\chi,tor}(X,\omega)$ for some function $\chi\in\W$. By Proposition \ref{P:Echitor} following this proof, $G_\f \in L_{\chi}(P)$, so $G_\f<+\infty$ a.e. on $P$. Since $G_\f$ is convex, this implies that $G_\f(s)<+\infty$ for all $s\in\inte P$. The implication $(ii)\Rightarrow (iii)$ follows from Lemma \ref{L:conv5}.

We next prove that $(iii)\Rightarrow (i)$. Consider the measure $\langle\omega_{\f}^{n}\rangle$ defined as the non-pluripolar product of the positive closed currents $\omega_\f:=\omega + dd^c\f$ \cite[Definition 1.1]{BEGZ10}. As $\f$ is locally bounded on $(\mathbb{C}^*)^n$, the Bedford-Taylor product $\omega^{n}_\f=\omega_\f\wedge\ldots\wedge\omega_\f$ is well defined on $(\mathbb{C}^*)^n$ \cite{BT76,BT82}. Since $(\mathbb{C}^*)^n=X\setminus A$, where $A$ is an analytic subset of $X$, it follows from \cite[p.\ 204, Proposition 1.6]{BEGZ10} that $\langle\omega^{n}_\f\rangle$ is the trivial extension of $\omega^{n}_\f$ to $X$. Then
\begin{eqnarray*}
\int_{X}\langle\omega^{n}_\f\rangle=\int_{(\C^{*})^{n}}\omega_\f^{n}&=&\int_{(\C^{*})^{n}}(dd^cF_\f\circ L)^n \\
&=& \int_{\R^{n}}MA_\R(F_{\f})=\int_{\R^{n}}MA_\R(F_0)=\int_{X}\omega^{n}\,.
\end{eqnarray*}

Therefore $\langle\omega^{n}_\f\rangle$ has full mass, so $\f\in \E_{tor}(X,\omega)$ and $MA(\varphi)=\langle\omega_{\f}^{n}\rangle$ \cite[Sect.\ 2]{BEGZ10}.

To show that $(ii)\Rightarrow(iv)$, let $\varepsilon>0$. Using Lemmas \ref{L:conv5} and \ref{L:F0FP} we get
$$\varphi=(F_\f-F_0)\circ L\geq-\varepsilon F_P\circ L-C-M_\varepsilon\,,$$
with some constants $C,M_\varepsilon>0$. By Lemma \ref{L:conv4} there exists a constant $a>0$ such that $F_P(x)\leq a\max\{|x_1|,\ldots,|x_n|\}$. These imply that
$$\varphi(z)\geq-2a\varepsilon\max\big\{\big|\log|z_1|\big|,\ldots,\big|\log|z_n|\big|\big\}$$
for $z\in(\C^\star)^n\setminus K_\varepsilon$,  where $K_\varepsilon=\{\varepsilon F_P\circ L\leq C+M_\varepsilon\}$.

Conversely, to prove that $(iv)\Rightarrow(ii)$, we let $\varepsilon\in(0,1)$ and by applying Lemma \ref{L:conv5} we need to show that there exists $M_\varepsilon>0$ such that $F_\f\geq(1-\varepsilon)F_P-M_\varepsilon$. By Lemma \ref{L:conv4} we have $F_P(x)\geq b\max\{|x_1|,\ldots,|x_n|\}$ for some constant $b>0$. Using $(iv)$ and Lemma \ref{L:F0FP} we obtain that
\begin{align*}
F_\f(L(z))&=F_0(L(z))+\varphi(z) \\
&\geq F_P(L(z))-C-b\varepsilon\max\big\{\big|\log|z_1|\big|,\ldots,\big|\log|z_n|\big|\big\}\\
&\geq(1-\varepsilon)F_P(L(z))-C\,,
\end{align*}
for $z\in(\C^\star)^n\setminus K_\varepsilon$, where $K_\varepsilon\subset(\C^\star)^n$ is a compact set. Since $F_\f,F_P$ are continuous this implies that $F_\f(L(z))\geq(1-\varepsilon)F_P(L(z))-M_\varepsilon$ on $(\C^\star)^n$, for some constant $M_\varepsilon>C$.

Recall that functions in the class $\E(X,\omega)$ have zero Lelong number at each point. To complete the proof we assume that $\f \in PSH_{tor}(X,\omega)$ has zero Lelong number at all the toric points of $X$ and show that $(iv)$ holds.

Let $\varepsilon>0$ and let $p_1,\ldots,p_N$ be the toric points of $X$. We denote as before by $z=(z_1,\ldots,z_n)$ the coordinates on the complex torus $(\C^\star)^n\subset X$. For $1\leq j\leq N$, there exists an open set $p_j\in V_j\subset X$ and a biholomorphic map $\Phi_j:V_j\to\C^n$ such that $\Phi_j(p_j)=0$, $(\C^\star)^n\subset V_j$, $\Phi_j((\C^\star)^n)=(\C^\star)^n$, and $X=V_1\cup\ldots\cup V_N$ (see e.g. \cite[Proposition 4.4]{ALZ} and its proof). Moreover, if $\Phi_j(z)=\zeta=(\zeta_1,\ldots,\zeta_n)$ then
$$(\log|\zeta_1|,\ldots,\log|\zeta_n|)=A_j(\log|z_1|,\ldots,\log|z_n|)\,,\,\text{ where } A_j\in GL_n(\Z)\,.$$
We denote by $\|A_j\|_\infty$ the operator norm of $A_j$ with respect to the sup norm (i.e. $\|A_jx\|_\infty\leq\|A_j\|_\infty\|x\|_\infty$) and let $\gamma:=\max\{\|A_1\|_\infty,\ldots,\|A_N\|_\infty\}$. Since $X$ is compact we can find $R>0$ such that $X=\bigcup_{j=1}^N\Phi_j^{-1}\big(\Delta^n(0,R)\big)$, where $\Delta^n(0,R)\subset\C^n$ is the open polydisc of radius $R$ centered at $0$.

We have that $\big(\Phi_j^{-1}\big)^\star\big(\omega\mid_{(\C^\star)^n}\big)=dd^cF_0^j\circ L$, where $L(\zeta_1,\ldots,\zeta_n)=(\log|\zeta_1|,\ldots,\log|\zeta_n|)$ and $F_0^j:\R^n\to\R$ is a smooth strictly convex function such that $F_0^j\circ L$ extends to a smooth psh function on $\C^n$. It follows that $\nabla F_0^j(\R^n)\subset(0,+\infty)^n$. Moreover, there exists a convex function $F_\f^j:\R^n\to\R$ such that $\nabla F_\f^j(\R^n)\subset[0,+\infty)^n$ and $F_\f^j\circ L=F_0^j\circ L+\varphi\circ\Phi_j^{-1}$ on $(\C^\star)^n$. The function $F_\f^j\circ L$ extends to a psh function on $\C^n$ and has Lelong number $\nu(F_\f^j\circ L,0)=0$, since Lelong numbers are invariant under biholomorphic maps. Hence there exists $r_j=r_j(\varepsilon)>0$ such that $F_\f^j(\log r,\ldots,\log r)\geq\frac{\varepsilon}{2}\,\log r$ for $0<r<r_j$. Since $F_\f^j$ is increasing in each variables this implies
$$
F_\f^j(\log|\zeta_1|,\ldots,\log|\zeta_n|)\geq\frac{\varepsilon}{2}\,\min\{\log|\zeta_1|,\ldots,\log|\zeta_n|\}\,,
$$
$\,\text{ if } \min\{|\zeta_1|,\ldots,|\zeta_n|\}<r_j\,.$
So there exists a constant $M_\varepsilon>0$ such that
$$
\varphi\circ\Phi_j^{-1}(\zeta)=F_\f^j\circ L(\zeta)-F_0^j\circ L(\zeta)\geq\frac{\varepsilon}{2}\,\min\{\log|\zeta_1|,\ldots,\log|\zeta_n|\}-M_\varepsilon\,,
$$
for $\zeta\in\Delta^n(0,R)$ with $\min\{|\zeta_1|,\ldots,|\zeta_n|\}<r_j$.
By shrinking $r_j$ we obtain
$$
\varphi\circ\Phi_j^{-1}(\zeta)\geq\varepsilon\min\{\log|\zeta_1|,\ldots,\log|\zeta_n|\}=-\varepsilon\max\big\{\big|\log|\zeta_1|\big|,\ldots,\big|\log|\zeta_n|\big|\big\},
$$
for $\zeta\in\Delta^n(0,R)$ with $\min\{|\zeta_1|,\ldots,|\zeta_n|\}<r_j$. It follows that
$$\varphi(z)\geq-\gamma\varepsilon\max\big\{\big|\log|z_1|\big|,\ldots,\big|\log|z_n|\big|\big\},$$
if $z\in U_j:=(\C^\star)^n\cap\Phi_j^{-1}\big(\Delta^n(0,R)\cap\{\zeta\in\C^n:\,\min\{|\zeta_1|,\ldots,|\zeta_n|\}<r_j\}\big)$. We note that $U_\varepsilon:=\bigcup_{j=1}^NU_j\subset(\C^\star)^n$ is open and $K_\varepsilon:=(\C^\star)^n\setminus U_\varepsilon$ is compact. Moreover the above lower estimate on $\varphi$ holds on $(\C^\star)^n\setminus K_\varepsilon$. This concludes the proof.
\end{proof}

\begin{Corollary}\label{C:Etor} If $\f\in\E_{tor}(X,\omega)$ and $\chi$ is a nonnegative continuous function on $\R^n$ then
$$\int_{(\C^\star)^n}(\chi\circ L)\,(dd^c F_\f\circ L)^n=\int_{\R^n} \chi\,MA_\R(F_\f)=n!\int_{\inte P}\chi(\nabla G_\f(s))\,dV(s).$$
\end{Corollary}

\begin{proof} The first equality follows from Lemma \ref{L:MAR}. The second one follows from \cite[Lemma 2.7]{BerBer13}, since $F_\f$ has full Monge-Amp\`ere mass by Theorem \ref{T:Etor}.
\end{proof}

 \begin{Examples}\label{E:Pn}
 If $X=\P^n$ is the complex projective space and $\omega$ is the Fubini-Study K\"ahler form, then
$F_0(x)=\frac{1}{2}\,\log \left( 1+\sum_{i=1}^n e^{2x_i} \right)$ and $P=\overline{\nabla F_0(\R^n)}$ is the simplex
$$
P=\left\{s_i \geq 0, \, 1 \leq i \leq n,\;\sum_{i=1}^n s_i \leq 1\right\}.
$$
Thus $d=n+1$, $\ell_i(s)=s_i$ for $1 \leq i \leq n$, and $\ell_{n+1}(s)=1-\sum_{i=1}^n s_i$.
The Legendre transform of $F_0$ is
$$
G_0(s)=\frac{1}{2} \left[ \sum_{i=1}^n s_i \log s_i+\left(1-\sum_{j=1}^n s_j \right)\log\left( 1-\sum_{j=1}^n s_j \right) \right].
$$
This coincides with the function given by Guillemin's formula.

 \smallskip

 1)  Let $[z]=[z_0:z_1:\ldots:z_n]$ denote the homogeneous coordinates on $\P^n$. The function
 $$\f_1[z]=\log|z_1|-\log\|z\|$$
 is $\omega$-psh and toric. It does not belong to the class $\E_{tor}(X,\omega)$
 since it has positive Lelong numbers along the toric hyperplane $(z_1=0)$.
 The associated convex function $F_1:\R^n \rightarrow \R$
 is given by $F_1(x)=x_1$ and   its Legendre transform is
 $$
 G_1(s)=0 \text{ if } s=(1,0,\ldots,0),
 $$
 and $G_1(s)=+\infty$ otherwise.

  \smallskip

 2) The function
 $$\f_2[z]=\max_{1 \leq i \leq n} \log |z_i|-\log\|z\|$$
  is $\omega$-psh and toric. It does not belong to the class $\E_{tor}(X,\omega)$
 since it has one positive Lelong number at the point $[1:0:\cdots:0]$. The corresponding convex function
 is $F_2(x)=\max_{1 \leq i \leq n} x_i$ and its Legendre transform is
 $$
 G_2(s)=0 \text{ if } s \in \left\{s_i \geq 0, \, 1 \leq i \leq n,\;\sum_{i=1}^n s_i=1\right\},
 $$
  and $G_2(s)=+\infty$ otherwise.
 \end{Examples}

\subsection{The classes ${\mathcal E}_{\chi,tor}(X,\omega)$}

If $\chi\in\W$ we define $L_{\chi}(P)$ to be the set of lower semicontinuous functions $G:P\to\R\cup\{+\infty\}$ such that
$$\int_P-\chi\big(\min_PG-G(s)\big)\,dV(s)<+\infty.$$

\begin{Proposition} \label{P:Echitor}
Let $\f \in PSH_{tor}(X,\omega)$. If $\chi\in\W$ and $\f \in \E_{\chi,tor}(X,\omega)$ then $G_\f \in L_{\chi}(P)$. Conversely, if $p\geq1$ and $G_\f \in L^p(P)$ then $\f \in \E^p_{tor}(X,\omega)$.
\end{Proposition}

\begin{proof} For the first claim, let $\f \in PSH_{tor}(X,\omega)\cap L^\infty(X)$ be such that $F_\f$ is smooth and strictly convex, and $F_\f\leq F_P\leq F_0$. We prove the following a priori estimate:
$$\int_{\inte P}-\chi(-G_\f(s))\,dV(s)\leq\frac{1}{n!}\,\int_X-\chi(\varphi)\,MA(\varphi)\,.$$
Note that $\f\leq0$ on $X$ and $G_\f\geq G_P=0$ on $P$. Moreover, by Proposition \ref{P:bdtorpsh} and Lemma \ref{L:conv3}, $\nabla F_\f:\R^n\to\inte P$ is bijective. Since $F_0\geq F_P$ and $F_\f(x)=\langle x,s\rangle-G_\f(s)$ for $x=\nabla G_\f(s)$, we obtain
\begin{eqnarray*}
(F_0-F_\f)\circ\nabla G_\f(s) &\geq & (F_P-F_\f)\circ\nabla G_\f(s) \\
& =&F_P(\nabla G_\f(s))-\langle \nabla G_\f(s),s\rangle+G_\f(s)\geq G_\f(s)\geq0,
\end{eqnarray*}
where $s\in\inte P$ and the last estimate follows from the definition of $F_P$. Applying Lemmas \ref{L:conv3} and \ref{L:MAR} we get
\begin{align*}
\int_{\inte P}-\chi(-G_\f(s))\,dV(s)&\leq\int_{\inte P}-\chi((F_\f-F_0)\circ\nabla G_\f(s))\,dV(s)\\
&=\frac{1}{n!}\,\int_{\R^n}-\chi(F_\f-F_0)\,MA_\R(F_\f) \\
&=\frac{1}{n!}\,\int_{(\C^\star)^n}-\chi(\f)\,MA(\f)\,.
\end{align*}

Let now $\f \in \E_{\chi,tor}(X,\omega)$ be such that $F_\f\leq F_P-1$. There exists a sequence $\f_j\in PSH_{tor}(X,\omega)\cap L^\infty(X)$ such that $\f_j\searrow\f$, the associated functions $F_{\f_j}$ are smooth and strictly convex, and $F_{\f_j}\leq F_P$. Then
$$
\int_X-\chi(\f_j)\,MA(\f_j)\to\int_X-\chi(\f)\,MA(\f)
$$
as $j\to+\infty$. Since $G_{\f_j}\nearrow G_\f$ it follows by the a priori estimate applied to $\f_j$ and the monotone convergence theorem that
$$\int_P-\chi(-G_\f(s))\,dV(s)\leq\frac{1}{n!}\,\int_X-\chi(\varphi)\,MA(\varphi)<+\infty\,,$$
so $G_\f \in L_{\chi}(P)$. This concludes the proof of the first claim.

\medskip

Conversely, let $p\geq1$ and consider the space of K\"ahler potentials $\mathcal H=\{\f\in \mathcal{C}^{\infty}(X): \,\omega+dd^c\f>0\}$ endowed with the metric
\[d_p(\f_1,\f_2)=\inf\int_{0}^{1}\left(\int_{X}|\dot{\f_t}|^p\,MA(\f_t)\right)^{1/p}dt\,,\,\;\f_1,\f_2\in\mathcal H,\]
where the infimum is taken over all smooth paths $t\in[0,1]\to\f_t\in\mathcal H$ joining $\f_1$ to $\f_2$. It is shown in \cite[Theorem 3]{Dar15} that if $\f_1,\f_2\in\mathcal H$ then
\begin{equation}\label{e:Dar15}
\int_{X}|\f_1-\f_2|^p\,MA(\f_1)\leq C_pd_p(\f_1,\f_2)^p,
\end{equation}
for some constant $C_p>1$ depending on $p$. On the other hand if $\f_1,\f_2\in\mathcal H\cap PSH_{tor}(X,\omega)$ are determined by the convex functions $F_1,F_2$ with Legendre transforms $G_1,G_2$, then, by \cite[Proposition 4.3]{G14},
\begin{equation}\label{e:V14}
d_p(\f_1,\f_2)^p=\int_P|G_1-G_2|^p\,dV\,.
\end{equation}
If $\f\in\mathcal H\cap PSH_{tor}(X,\omega)$ we apply \eqref{e:Dar15} and \eqref{e:V14} with $\f_1=\f$ and $\f_2=0$, and obtain that
\[\int_{X}|\f|^p\,MA(\f) \leq C_p\int_P|G_\f-G_0|^p\,dV\leq 2^{p-1}C_p\left(\|G_\f\|^p_{L^p(P)}+\|G_0\|^p_{L^p(P)} \right).\]

Let now $\f\in PSH_{tor}(X,\omega)$ be such that $G_\f \in L^p(P)$, and take a sequence $\f_j\in\mathcal H\cap PSH_{tor}(X,\omega)$ such that $\f_j\searrow\f$. Then $G_{\f_j}\nearrow G_\f$, so by the above estimate applied to $\f_j$ and dominated convergence we conclude that $\sup_j \int_X |\f_j|^p\,MA(\f_j) <+\infty$. Hence $\f \in \E^p_{tor}(X,\omega)$.
\end{proof}

\begin{Example} \label{E:proj}
Let $X=\P^1$ and $\omega$ be the Fubini-Study K\"ahler form. By Examples \ref{E:Pn} the corresponding convex function is
$F_0(x)=\frac{1}{2}\,\log \left( 1+e^{2x} \right)$ and $P=\overline{F_0'(\R)}=[0,1]$. Let $\f$ be the toric $\omega$-sh function associated to the convex function $F(x):=F_P(x)=\max(x,0)$. Note that the Legendre transform of $F$ is $G=0$ on $[0,1]$, and $dd^c F(\log|z|)$ is the (normalized) Lebesgue measure on the unit circle $S^1  \subset \P^1$.
We consider the sequence of toric $\omega$-sh $\{\f_j$\} defined by the convex functions
$$
F_j(x)=(1-\e_j) F(x)+\e_j \max(x,-C_j),
$$
where $\e_j$ decreases to $0$, while $C_j$ increases to $+\infty$. A straightforward computation yields that the corresponding Legendre transforms are
$$
G_j(s)=\max\big( C_j(\e_j-s), 0\big),\;0\leq s\leq1.
$$
Note that
\begin{equation*}
\begin{split}
\f_j(z)-\f(z)&=-\e_j\log^+|z|+\e_j\max(\log|z|,-C_j)\\
&=\begin{cases}
-\e_jC_j, &\quad |z|<e^{-C_j},\\
\e_j\log|z|, &\quad e^{-C_j}\leq|z|<1,\\
0, &\quad |z|\geq1.
\end{cases}
\end{split}
\end{equation*}
Thus we obtain the following:
\begin{itemize}
\item $\f_j \longrightarrow \f$ in $L^1$ if and only if $\e_j \rightarrow 0$;
\item $\f_j \longrightarrow \f$ in $L^{\infty}$ if and only if $\e_j C_j\rightarrow 0$;
\item $\f_j \longrightarrow \f$ in $W^{1,2}$ (the natural topology on $\E^1(X,\omega)$)
if and only if $\e_j^{2} C_j \rightarrow 0$.
\end{itemize}
\end{Example}

\subsection{Finite moments}

It is tempting to think that one can characterize the condition $\f \in \E_{tor}^q(X,\omega)$
by a finite moment condition, as follows. Let $\f \in \E_{tor}(X,\omega)$, $\mu_{\R}(\f):=MA_\R(F_\f)$,
$0<q<n$, and $q^*=nq/(n-q)$ denote its Sobolev conjugate exponent.
Does one have
$$
\f \in \E_{tor}^{q^*}(X,\omega) \Longleftrightarrow \int_{\R^n} \|x\|^q\, d \mu_\R(\f) <+\infty\; ?
$$

This question was raised by E. Di Nezza, who showed in \cite[Proposition 2.5]{DiN15} that, if $\f \in \E_{tor}(X,\omega)$, $n \geq 2$ and $1\leq q<n$, then
 $$
 \int_{\R^{n}} \|x\|^q \,d\mu_\R(\f) <+\infty \Longrightarrow \f \in \E_{tor}^{q^*}(X,\omega).
 $$

We have the following partial answer to this question in dimension $n=1$, i.e. when $X=\P^1$ and $\omega=\omega_{FS}$:

\begin{Proposition} \label{P:moment}
Let $(X,\om)=(\P^1,\om_{FS})$, $\f \in \E_{tor}(X,\omega)$, and $0<q<1$.

(i) If $q\geq1/2$ and $\int_{\R} |x|^q\,d \mu_\R(\f) <+\infty$ then $\f \in \E_{tor}^{q^*}(X,\omega)$.

\smallskip

(ii) If $\f \in \E_{tor}^{1}(X,\omega)$ then $\int_{\R} |x|^q\,d\mu_\R(\f)<+\infty$ for all $q<1/2$.

\smallskip

(iii) There exists a function $\f \in \E_{tor}^{1}(X,\omega)$ with  $\int_{\R} |x|^{1/2}\,d\mu_\R(\f) =+\infty$.
\end{Proposition}

\begin{proof} Recall that in this case $P=[0,1]$ (see Example \ref{E:proj}).

$(i)$ Replacing $\f$ by $\f+C$ we may assume that $\min_{[0,1]}G_\f=G_\f(a)=0$ for some $a\in[0,1]$. Note that $G_\f$ is convex and finite, so it is differentiable a.e. on $(0,1)$. Let $s,t\in(0,1)$ be such that $G_\f'(s),G_\f'(t)$ exist and $t$ is between $a$ and $s$. Since $G_\f$ is convex and it assumes its minimum at $a$ we have $|G_\f'(t)|\leq|G_\f'(s)|$. It follows that
\begin{equation*}
\begin{split}
0\leq G_\f(s)&=\left|\int_a^sG_\f'(t)\,dt\right|\leq\left|\int_a^s|G_\f'(t)|^{1-q}|G_\f'(t)|^q\,dt\right|\\
&\leq|G_\f'(s)|^{1-q}\left|\int_a^s|G_\f'(t)|^q\,dt\right|\leq|G_\f'(s)|^{1-q}\int_0^1|G_\f'(t)|^q\,dt.
\end{split}
\end{equation*}
Using Corollary \ref{C:Etor} we obtain
\begin{equation*}
\begin{split}
\int_0^1G_\f(s)^{\frac{q}{1-q}}\,ds&\leq\left(\int_0^1|G_\f'(t)|^q\,dt\right)^{\frac{q}{1-q}}\int_0^1|G_\f'(s)|^q\,ds\\
&=\left(\int_0^1|G_\f'(s)|^q\,ds\right)^{\frac{1}{1-q}}=\left(\int_{\R} |x|^q\,d \mu_\R(\f)\right)^{\frac{1}{1-q}}<+\infty.
\end{split}
\end{equation*}
Since $q/(1-q)\geq1$, Proposition \ref{P:Echitor} yields that $\varphi\in \mathcal{E}_{tor}^{\frac{q}{1-q}}(X,\omega)$.

\smallskip

$(ii)$ Let $\f \in \E_{tor}^{1}(X,\omega)$ and $q<1/2$. Then $G_\f\in L^1(P)$ by Proposition \ref{P:Echitor}.  The conclusion follows by showing that if $F_\f\leq F_P$ on $\R$ then
\begin{equation*}
\int_{\R} |x|^q\,d\mu_\R(\f)\leq \frac{2(1-q)}{1-2q}\,\|G_\f\|_{L^1}^{q}.
\end{equation*}
Note that it suffices to prove this in the case when $\f$ is bounded. Indeed, if $\f \in \E_{tor}^{1}(X,\omega)$ is such that $F_\f\leq F_P$ on $\R$, then there exists a sequence of bounded toric $\omega$-psh functions $\f_j\searrow\f$ such that $F_{\f_j}\leq F_P$ on $\R$. Hence $0\leq G_{\f_j}\nearrow G_\f$. Since $\mu_\R(\f_j)\to\mu_\R(\f)$ weakly on $\R$ it follows by the monotone convergence theorem that
\[\int_{\R} |x|^q\,d\mu_\R(\f)\leq\liminf_{j\to\infty}\int_{\R} |x|^q\,d\mu_\R(\f_j)\leq\frac{2(1-q)}{1-2q}\,\|G_\f\|_{L^1}^{q}.\]

Assume that $\f$ is a bounded toric $\omega$-psh function such that $F_\f\leq F_P$. Then $G_\f\geq0$ is a continuous convex function on $[0,1]$, and we fix $a\in[0,1]$ such that $\min_{[0,1]}G_\f=G_\f(a)\geq0$. Applying H\"older's inequality with $p=1/(1-q)$ we get, since $1/(1-pq)=(1-q)/(1-2q)>1$, that
\begin{equation*}
\begin{split}
\int_0^a|G'_\f(s)|^q\,ds & = \int_0^as^{-q}(-sG'_\f(s))^q\,ds\\
& \leq\left(\int_0^a(-sG'_\f(s))\,ds\right)^q\left(\int_0^as^{-pq}\,ds\right)^{\frac{1}{p}}\\
& \leq \left(-aG_\f(a)+\int_0^aG_\f(s)\,ds\right)^q\left(\int_0^1s^{-pq}\,ds\right)^{\frac{1}{p}}\\
& \leq \frac{1-q}{1-2q}\,\|G_{\f}\|_{L^1}^q\,.
\end{split}
\end{equation*}
Similarly,
\begin{equation*}
\begin{split}
\int_a^1|G'_\f(s)|^q\,ds & = \int_a^1(1-s)^{-q}((1-s)G'_\f(s))^q\,ds\\
& \leq\left(\int_a^1(1-s)G'_\f(s)\,ds\right)^q\left(\int_a^1(1-s)^{-pq}\,ds\right)^{\frac{1}{p}}\\
& \leq \left(-(1-a)G_\f(a)+\int_a^1G_\f(s)\,ds\right)^q\left(\int_0^1(1-s)^{-pq}\,ds\right)^{\frac{1}{p}}\\
& \leq \frac{1-q}{1-2q}\,\|G_{\f}\|_{L^1}^q\,.
\end{split}
\end{equation*}
Using the last two estimates and Corollary \ref{C:Etor} we obtain
\[\int_{\R} |x|^q\,d\mu_\R(\f)=\int_0^1|G'_\f(s)|^q\,ds\leq\frac{2(1-q)}{1-2q}\,\|G_{\f}\|_{L^1}^q\,.\]

\smallskip

$(iii)$ Let $\f \in PSH_{tor}(X,\omega)$ be determined by a convex function $F_\f$ defined as follows on the prescribed intervals and smooth on $\R$:
$$
F_\f(x)=\begin{cases}
x-\frac{2\sqrt{x}}{\ln x} & \text{if } x\geq e^3,\\
0 & \text{if } x \leq 0 .
\end{cases}
$$
Note that $F_\f(x)\leq x\leq F_0(x)=\frac{1}{2}\,\log \left( 1+e^{2x} \right)$ for $x\geq0$, and $\f \in \E_{tor}(X,\omega)$ since $F_\f$ has full Monge-Amp\`ere mass. Moreover,
$$\frac{1}{18x^{3/2}\ln x}\leq F''_\f(x)=\frac{1-8(\ln x)^{-2}}{2x^{3/2}\ln x}\leq\frac{1}{2x^{3/2}\ln x}\,,\,\text{ for }x\geq e^3.$$
Therefore
\[\int_\R|x|^{1/2}F''_\f(x)\,dx\geq\frac{1}{18}\,\int_{e^3}^{+\infty}\frac{1}{x\ln x}\,dx=+\infty.\]
Since $\f \in \E_{tor}(X,\omega)$ the measure $MA(\f)$ does not charge polar sets. Hence
\begin{equation*}
\begin{split}
\int_X(-\f)\,MA(\f) & =\int_\R(F_0-F_\f)F''_\f \leq C+\int_{e^3}^{+\infty}\frac{2\sqrt{x}}{\ln x}\,F''_\f(x)\,dx \\
 & \leq C+\int_{e^3}^{+\infty}\frac{1}{x(\ln x)^2}\,dx<+\infty,
\end{split}
\end{equation*}
for some constant $C$, which implies that  $\f \in \E^1_{tor}(X,\omega)$.
\end{proof}

\section{Higher regularity}\label{S:LogLip}

\subsection{Continuous toric functions}
These can be characterized as follows:

\begin{Proposition}\label{P:bdtorpsh}
Let $\varphi\in PSH_{tor}(X,\omega)$. The following are equivalent:

(i) $\varphi$ is continuous on $X$;

(ii) $\varphi\in L^\infty(X)$;

(iii) $F_P-C\leq F_\varphi\leq F_P+C$ for some constant $C\geq0$;

(iv) $G_P-C\leq G_\varphi\leq G_P+C$ for some constant $C\geq0$.\\
Moreover, we have in this case $\|F_\varphi-F_P\|_{L^\infty(\R^n)}=\|G_\varphi\|_{L^\infty(P)}$.
\end{Proposition}

\begin{proof} Assume that $\varphi \in PSH_{tor}(X,\omega)$ is bounded. Using the notation from the proof of Theorem \ref{T:Etor} we let $p_1,\ldots,p_N$ be the toric points of $X$ and $p_j\in V_j\subset X$ be open sets with biholomorphic maps $\Phi_j:V_j\to\C^n$ such that $\Phi_j(p_j)=0$, $(\C^\star)^n\subset V_j$, $\Phi_j((\C^\star)^n)=(\C^\star)^n$, $X=V_1\cup\ldots\cup V_N$. If $L(\zeta)=(\log|\zeta_1|,\ldots,\log|\zeta_n|)$, $\zeta\in\C^n=\Phi_j(V_j)$, there exist convex functions $F_\f^j,F_0^j:\R^n\to\R$ such that $F_\f^j\circ L,F_0^j\circ L$ extend to a psh function, respectively to a smooth psh function, on $\C^n$, and $F_\f^j\circ L=F_0^j\circ L+\varphi\circ\Phi_j^{-1}$ on $(\C^\star)^n$. Since $\f$ is bounded and polyradial psh functions on $\C^n$ are continuous, this shows that $\varphi$ is continuous on each $V_j$, hence on $X$.

Using Lemma \ref{L:F0FP} we see immediately that $(ii)\Leftrightarrow(iii)$, while $(iii)\Leftrightarrow(iv)$ follows from the definition of the Legendre transform. Moreover, if $(iii)$ holds with a constant $C$ then $(iv)$ holds with the same constant, and vice versa. This implies the last claim.
\end{proof}

We note that assertion $(iii)$ in Proposition \ref{P:bdtorpsh} is equivalent to the condition that $G=+\infty$ on $\R^n\setminus P$ and $G$ is bounded above on $P$.

\subsection{Log-Lipschitz Legendre transforms}

Recall that a continuous function $u:\Omega \subset \R^n \rightarrow \R$ is Log-Lipschitz
if its modulus of continuity $\omega_u(x,r)$ is locally bounded from above by $C r \log r$.

In order to prove Theorem C we need the following preliminary results.

\begin{Lemma}\label{L:integral} Let $n\geq1$ and
\[I(\lambda)=\int_0^1(t^{n-1}+\lambda^{-1})\log(1+\lambda^{-1}t^{-n+1})\,dt\,,\,\;\lambda>0.\]
If $0<x\leq1/e$ and $\lambda_x=(n+3)x\log\frac{1}{x}$ then $xI(\lambda_x)<1$.
\end{Lemma}

\begin{proof}
We have
\[I(\lambda) \leq \int_0^1(t^{n-1}+\lambda^{-1})\log\frac{1+\lambda}{\lambda t^{n-1}}\,dt=\left(\frac{1}{\lambda}+\frac{1}{n}\right)\log\frac{1+\lambda}{\lambda}+\frac{n-1}{\lambda}+\frac{n-1}{n^2}\,.\]
Since $x\leq 1/e$ we have that $1+\lambda_x\leq 1+(n+3)/e$ and
\[\log\frac{1+\lambda_x}{\lambda_x}\leq\log\frac{1}{x}+\log\left(\frac{1}{n+3}+\frac{1}{e}\right)-\log\log\frac{1}{x}<\log\frac{1}{x}\,.\]
Therefore
\begin{equation*}
\begin{split}
I(\lambda_x) &< \left(\frac{1}{\lambda_x}+\frac{1}{n}\right)\log\frac{1}{x}+\frac{n-1}{\lambda_x}+\frac{1}{n}\\
&= \frac{1}{(n+3)x}+\frac{(n-1)}{(n+3)x\log\frac{1}{x}}+\frac{1}{n}\log\frac{1}{x}+\frac{1}{n}\\
&\leq \frac{1}{x}\left(\frac{n}{n+3}+\frac{x}{n}\log\frac{1}{x}+\frac{x}{n}\right)\leq \frac{1}{x}\left(\frac{n}{n+3}+\frac{2}{ne}\right)<\frac{1}{x}\,.
\end{split}
\end{equation*}
\end{proof}

\begin{Proposition}\label{P:Log-Lip}
Let $P\subset\R^n$ be a compact convex polytope and $f:\inte P\to\R$ be a locally Lipschitz function. If $e^{\varepsilon\|\nabla f\|}\in L^1(P)$ for some $\varepsilon>0$ then $f$ extends to a Log-Lipschitz function on $P$.
\end{Proposition}

\begin{proof}
If $n=1$ then $P=[a,b]$, and for $a<s_1<s_2<b$ it follows by Jensen's inequality that
\begin{equation*}
\begin{split}
|f(s_2)-f(s_1)| & \leq \frac{1}{\varepsilon}\int_{s_1}^{s_2}\varepsilon|f'(t)|\,dt\leq \frac{s_2-s_1}{\varepsilon}\,\log\left(\frac{1}{s_2-s_1}\,\int_{s_1}^{s_2}e^{\varepsilon|f'(t|}\,dt\right)\\
& \leq \frac{s_2-s_1}{\varepsilon}\,\log\frac{\|e^{\varepsilon|f'|}\|_{L^1[a,b]}}{s_2-s_1}\,.
\end{split}
\end{equation*}

We consider next the case $n>1$. Since $P$ is convex there exists a constant $c\in(0,1)$ with the following property: for every $s_1,s_2\in P$ there exists a compact subset $A$ of the hyperplane perpendicular to the segment $[s_1,s_2]$ at its midpoint such that $A\subset\inte P$ and
\begin{equation}\label{e:cones}
c\|s_1-s_2\|^{n-1}\leq V_{n-1}(A)\leq1,\,\;\|s_1-\sigma\|\leq\frac{\|s_1-s_2\|}{2c}\text{ for all } \sigma\in A,
\end{equation}
where $V_{n-1}(A)$ is the $(n-1)$-dimensional Hausdorff measure of $A$. Note that we then have $\|s_2-\sigma\|\leq\frac{\|s_1-s_2\|}{2c}$ for all $\sigma\in A$.

We will show that if $s_1,s_2\in \inte P$ are such that $\|s_1-s_2\|\leq\frac{2}{e}$ then
\begin{equation}\label{e:Log-Lip1}
|f(s_1)-f(s_2)|\leq 2C\|s_1-s_2\|\log\frac{2}{c\|s_1-s_2\|^n}\,,
\end{equation}
where
\[C=\frac{n+3}{\varepsilon c}\,\max\left(1,\big\|e^{\varepsilon\|\nabla f\|}\big\|_{L^1(P)}\right).\]
This clearly implies that $f$ extends to a Log-Lipschitz function on $P$.

Fix $s_1,s_2\in \inte P$ with $\|s_1-s_2\|\leq\frac{2}{e}$ and let $A$ be a set as in \eqref{e:cones}. Note that \eqref{e:Log-Lip1} follows if we prove that
\begin{equation}\label{e:Log-Lip2}
\left|f(s_1)-\frac{1}{V_{n-1}(A)}\int_Af\,dV_{n-1}\right|\leq C\|s_1-s_2\|\log\frac{2}{c\|s_1-s_2\|^n}\,,
\end{equation}
since the same holds with $s_2$ in place of $s_1$.

We may assume that $s_1=(0,a)\in\R^{n-1}\times\R$, with $a>0$, and that $A=B\times\{0\}\subset\R^{n-1}\times\{0\}$. Then $\|s_1-s_2\|=2a$. We set $\sigma=(\sigma',0)\in A$ for $\sigma'\in B$.
Since $f$ is locally Lipschitz on $\inte P$ and, by \eqref{e:cones},  $\|s_1-\sigma\|\leq a/c$ for $\sigma\in A$, we obtain
\begin{equation}\label{e:Log-Lip3}
\begin{split}
 \big|V_{n-1}(B) & f(s_1)  - \int_Bf(\sigma)\,dV_{n-1}(\sigma')\big|=\\
& \big|\int_B\int_0^1\langle\nabla f((1-t)s_1+t\sigma),s_1-\sigma\rangle\,dt\,dV_{n-1}(\sigma')\big|\\
& \leq \int_B\int_0^1\|\nabla f((1-t)s_1+t\sigma)\|\,\|s_1-\sigma\|\,dt\,dV_{n-1}(\sigma')\\
& \leq\frac{a}{c}\,\int_B\int_0^1\|\nabla f((1-t)s_1+t\sigma)\|\,dt\,dV_{n-1}(\sigma')\\
& =\frac{1}{\varepsilon c}\,\int_B\int_0^1\varepsilon\|\nabla f((1-t)s_1+t\sigma)\|\,t^{-n+1}\,d\mu,
\end{split}
\end{equation}
where $\mu$ is the measure on $B\times[0,1]$ given by $d\mu=at^{n-1}dt\,dV_{n-1}$.

Consider the weight $\chi(x)=(x+1)\log(x+1)-x$, $x\geq0$, with conjugate weight (Legendre transform) $\chi^\star(y)=e^y-y-1$, $y\geq0$, and the Orlicz spaces $L^\chi(B\times[0,1],\mu)$ and $L^{\chi^\star}(B\times[0,1],\mu)$. Recall that the norm on $L^\chi(B\times[0,1],\mu)$ is given by
\[\|g\|_\chi:=\inf\left\{\lambda>0:\,\int_{B\times[0,1]}\chi(|g|/\lambda)\,d\mu\leq1\right\},\]
and one has that $\|g\|_\chi\leq\max\left(1,\int_{B\times[0,1]}\chi(|g|)\,d\mu\right)$.

Estimating the last integral in \eqref{e:Log-Lip3} by the multiplicative H\"older-Young inequality (see \cite[Proposition 2.15]{BBEGZ} or \cite{RR}) we get
\begin{equation}\label{e:Log-Lip4}
\begin{split}
 \big|V_{n-1}(B)f(s_1)  - & \int_Bf(\sigma)\,dV_{n-1}(\sigma')\big|\\
& \leq\frac{2}{\varepsilon c}\,\big\|\varepsilon\|\nabla f((1-t)s_1+t\sigma)\|\big\|_{\chi^\star}\|t^{-n+1}\|_\chi\,.
\end{split}
\end{equation}
If $\Gamma$ is the cone in $\R^n$ with vertex $s_1$ and base $A$ then
\[\int_\Gamma e^{\varepsilon\|\nabla f\|}\,dV_n=\int_{B\times[0,1]}e^{\varepsilon\|\nabla f((1-t)s_1+t\sigma)\|}\,d\mu.\]
Since $\chi^\star(y)<e^y$ it follows that
\begin{equation*}
\begin{split}
\big\|\varepsilon\|\nabla f((1-t)s_1+t\sigma)\|\big\|_{\chi^\star} & \leq \max\left(1,\int_{B\times[0,1]}e^{\varepsilon\|\nabla f((1-t)s_1+t\sigma)\|}\,d\mu\right)\\
& \leq \max\left(1,\int_Pe^{\varepsilon\|\nabla f\|}\,dV_n\right).
\end{split}
\end{equation*}

It remains to estimate the second Orlicz norm in \eqref{e:Log-Lip4}. We have
\begin{equation*}
\begin{split}
\int_{B\times[0,1]} & \chi(t^{-n+1}/\lambda)\,d\mu= \\
& =\int_B\int_0^1\left[\big(\frac{t^{-n+1}}{\lambda}+1\big)\log\big(\frac{t^{-n+1}}{\lambda}+1\big)-\frac{t^{-n+1}}{\lambda}\right]at^{n-1}\,dt\,dV_{n-1}\,, \\
& \leq aV_{n-1}(B)\,I(\lambda),
\end{split}
\end{equation*}
where $I(\lambda)$ is the function from Lemma \ref{L:integral}. Note that
\[aV_{n-1}(B)\leq\|s_1-s_2\|/2\leq1/e\,,\]
since $c\|s_1-s_2\|^{n-1}\leq V_{n-1}(B)=V_{n-1}(A)\leq1$ by \eqref{e:cones}. Lemma 4.2 implies that
\[aV_{n-1}(B)\,I(\lambda_0)\leq 1\,,\,\text{ if } \lambda_0=(n+3)aV_{n-1}(B)\log\frac{1}{aV_{n-1}(B)}\,,\]
hence
\begin{equation*}
\begin{split}
\|t^{-n+1}\|_\chi & =\inf\left\{\lambda>0:\,\int_{B\times[0,1]}\chi(t^{-n+1}/\lambda)\,d\mu\leq1\right\}\\
& \leq\lambda_0\leq(n+3)aV_{n-1}(B)\log\frac{1}{c\|s_1-s_2\|^{n-1}a}\,.
\end{split}
\end{equation*}

By \eqref{e:Log-Lip4} we conclude that
\begin{equation*}
\begin{split}
 \big|f(s_1)  - &\frac{1}{V_{n-1}(B)} \int_Bf(\sigma)\,dV_{n-1}(\sigma')\big|\leq\\
& \frac{2(n+3)}{\varepsilon c}\,\max\left(1,\big\|e^{\varepsilon\|\nabla f\|}\big\|_{L^1(P)}\right)a\log\frac{1}{c\|s_1-s_2\|^{n-1}a}\,.
\end{split}
\end{equation*}
This yields \eqref{e:Log-Lip2}, since $a=\|s_1-s_2\|/2$, and the proof is finished.
\end{proof}

We now prove Theorem C stated in the Introduction.

\begin{Theorem} \label{thm:holder}
Let $\f \in \E_{tor}(X,\omega)$. The following properties are equivalent:

(i) There exists $\e>0$ such that $\exp(-\e PSH_{tor}(X,\omega) ) \subset L^1(MA(\f))$;

\smallskip

(ii) There exists $\e>0$ such that $e^{\e \|\nabla G_\f\|}\in L^1(P)$;

\smallskip

(iii)  The function $G_\f$ is Log-Lipschitz on $P$.

\smallskip

(iv) There exists a constant $C>0$ such that
$\|\nabla G_\f(s)\|\leq C\log\frac{C}{\dist(s,\partial P)}$
holds for almost all $s\in\inte P$.
\end{Theorem}

Recall that Guillemin's potentials are only Log-Lipschitz continuous on the Delzant polytope $P$, although
they correspond to smooth toric $\omega$-psh functions on $X$.
The observation we make here is that this regularity actually corresponds to a class of toric $\omega$-psh functions which seem to be merely H\"older continuous on $X$ (see Remark \ref{rem:holder}).

\begin{proof}
We set $\mu_\R(\f):=MA_\R(F_\f)$. Since $\f \in \E_{tor}(X,\omega)$ the measure $MA(\f)$ does not charge pluripolar sets, so by Lemma \ref{L:MAR},
\[\int_Xe^{-\e\psi}\,MA(\f)=\int_{(\C^\star)^n}e^{-\e(F_\psi-F_0)\circ L}\,(dd^c F_\f\circ L)^n=\int_{\R^n} e^{-\e(F_\psi-F_0)}\,d\mu_\R(\f), \]
for every $\psi\in PSH_{tor}(X,\omega)$. Using Lemma \ref{L:F0FP} and Proposition \ref{P:torpsh}, it follows that $(i)$ is equivalent to condition

\smallskip

$(i')\;\displaystyle \int_{\R^n} e^{-\e(F-F_P)}\,d\mu_\R(\f) <+\infty\,,\,\;\forall\, F:\R^n \rightarrow \R \text{ convex function with } \\ F \leq F_P+O(1) \text{ on } \R^n.$

\smallskip

To show $(i')\Leftrightarrow(ii)$, we may assume that $0\in\inte P$ and we fix constants $a,b >0$ such that the closed balls $\ov B(0,b)\subset P \subset\ov B(0,a)$. Then by Lemma \ref{L:conv4}, $b\|x\| \leq F_P(x)  \leq a\|x\|$.
If $(i')$ holds, we apply it with $F=0$ to conclude by Corollary \ref{C:Etor} that
\[\int_{\inte P}e^{\e b\|\nabla G_\f(s)\|}\,dV(s)=\int_{\R^n}e^{\e b\|x\|}\,d\mu_\R(\f)\leq\int_{\R^n} e^{\e F_P(x)} \,d\mu_\R(\f) <+\infty,\]
which gives $(ii)$. Conversely, assume $(ii)$ holds and let $F$ be a function as in $(i')$. Then by Proposition \ref{P:torpsh}, $\nabla F(\R^n)\subset P\subset\ov B(0,a)$, so
\[F(x)=F(0)+\int_0^1\langle\nabla F(tx),x\rangle\,dt\geq-a\|x\|+F(0)\,.\]
Therefore $F_P(x)-F(x)\leq 2a\|x\|-F(0)$ and
\begin{equation*}
\begin{split}
\int_{\R^n}e^{-\frac{\e}{2a}\,(F-F_P)}\,d\mu_\R(\f) & \leq e^{-\frac{\e}{2a}\,F(0)}\,\int_{\R^n}e^{\e\|x\|}\,d\mu_\R(\f) \\
& =e^{-\frac{\e}{2a}\,F(0)}\,\int_{\inte P}e^{\e\|\nabla G_\f(s)\|}\,dV(s)<+\infty\,,
\end{split}
\end{equation*}
so $(i')$ holds.

Proposition \ref{P:Log-Lip} shows that $(ii)$ implies $(iii)$. We prove next that $(iii)$ implies $(iv)$. Since $G_\f$ is Log-Lipschitz on the compact polytope $P$ it follows that there exists a constant $C>0$ such that $\|s-s'\|\leq C/2$ and
\[|G_\f(s)-G_\f(s')|\leq C\|s-s'\|\log\frac{C}{\|s-s'\|}\,,\]
for all $s,s'\in P$. Let $s\in\inte P$ be such that $G_\f$ is differentiable at $s$, $\nabla G_\f(s)\neq0$, and let $\nu$ be the unit vector in the direction of $\nabla G_\f(s)$. We consider the convex function
\[g(t)=G_\f(s+t\nu),\;0\leq t\leq t^\star,\]
where $t^\star>0$ is defined such that $s^\star:=s+t^\star\nu\in\partial P$. Then $t^\star=\|s^\star-s\|\geq\dist(s,\partial P)$ and
\begin{equation*}
\begin{split}
\|\nabla G_\f(s)\| & =g'(0)\leq\frac{g(t^\star)-g(0)}{t^\star}=\frac{G_\f(s^\star)-G_\f(s)}{\|s^\star-s\|}\\
& \leq C\log\frac{C}{\|s^\star-s\|}\leq C\log\frac{C}{\dist(s,\partial P)}\,.
\end{split}
\end{equation*}

Finally, we note that $(iv)$ clearly implies that $(ii)$ holds with $\e>0$ small enough, and the proof is complete.
\end{proof}

\begin{Remark} \label{rem:holder}
It is tempting to think that these conditions are all equivalent to the fact that
$\f$ is H\"older continuous. This is easily seen to be the case when $n=1$.
We refer the interested reader to \cite{DDGHKZ} for more information, geometric motivations,
 and related questions connecting the
H\"older continuity of Monge-Amp\`ere potentials to the integrability properties of the associated
complex Monge-Amp\`ere measure.
\end{Remark}

\begin{Example}
Fix $0< \a<1$ and consider the convex function $F:\R \rightarrow \R$ defined
by $F(x)=e^{\a x}$ when $x \leq 0$ and $F(x)=x+1$ when $x \geq 0$.
It determines a H\"older continuous toric $\omega_{FS}$-psh function $\f$ on $\P^1$,
which is defined in $\C$ by
\[\f(z)=\begin{cases}
|z|^\a -\log \sqrt{1+|z|^2} & \text{ if } |z| \leq 1, \\
\log |z|+1 -\log \sqrt{1+|z|^2} & \text{ if } |z| \geq 1.
\end{cases}\]
We let the reader check that the Legendre transform of $F$ is given by
\[G(s)=\begin{cases}
\frac{s}{\a} \log\frac{s}{\a} -\frac{s}{\a} & \text{if } 0 \leq s \leq \a,\\
-1 & \text{if } \a \leq s \leq 1.
\end{cases}\]
\end{Example}

\bigskip
\smallskip

\noindent{\bf Aknowledgements :}{ This paper is dedicated to the memory of Professor J\'ozef Siciak, who started a systematical study of extremal plurisubharmonic functions in the early sixties, laying the first stones of Pluripotential Theory. His ideas have been very influential.}

\end{document}